\documentclass[11pt]{amsart}

\usepackage[T1]{fontenc}
\usepackage[utf8]{inputenc}
\usepackage[english]{babel}

\usepackage{amsmath,amssymb,amsthm,mathtools}
\usepackage{mathrsfs}

\usepackage[hidelinks]{hyperref}

\numberwithin{equation}{section}
\theoremstyle{plain}
\newtheorem{theorem}{Theorem}[section]
\newtheorem*{theorem*}{Theorem}
\newtheorem{proposition}[theorem]{Proposition}
\newtheorem{lemma}[theorem]{Lemma}
\newtheorem{corollary}[theorem]{Corollary}
\newtheorem{example}[theorem]{Example}

\newtheorem{conjecture}[theorem]{Conjecture}
\theoremstyle{definition}

\theoremstyle{remark}

\newcommand{\C}{\mathbb{C}}
\newcommand{\p}{\mathbb{P}}

\newcommand{\OO}{\mathcal{O}}
\newcommand{\F}{\mathscr{F}}
\newcommand{\G}{\mathscr{G}}

\newcommand{\Mer}{\mathscr M}

\DeclareMathOperator{\Int}{Int}

\title[Linear spaces of  rational integrable 1-forms]{Linear spaces of rational integrable 1-forms}
\author{Gabriel Santos Barbosa}
\address{Dipartimento di Matematica, Universit\`a degli Studi di Bari Aldo Moro, Bari, Italy}
\email{214gsb@gmail.com}
\date{}

\begin{document}

\begin{abstract}
We study finite-dimensional spaces of rational one-forms on a projective manifold by means of their integrable locus.
\end{abstract}

\maketitle

\section{Introduction}
Let $X$ be a projective manifold, and denote by $\Mer_X$ its sheaf of rational functions. A codimension $q$ distribution on $X$ is an equivalence class of non-zero rational decomposable $q$-forms, where two elements
$$
\omega, \omega'\in H^0(X,\Omega^q_X\otimes \Mer_X)
$$
define the same distribution if there exists a non-zero rational function $f$ such that $\omega=f\omega'$. A codimension $q$ distribution on $X$ is said to be a codimension $q$ foliation if it admits a representative
\(
\omega=\alpha_1\wedge\cdots\wedge\alpha_q
\)
satisfying the integrability equations
\[
\omega\wedge d\alpha_i=0, \qquad i=1,\ldots,q.
\]

Given a distribution $\mathcal D$ defined by the $q$-form $\omega$, one obtains a subsheaf
\(
T_{\mathcal D}\subset T_X
\)
whose local sections are the vector fields $v$ satisfying $\iota_v\omega=0$. The condition that $\mathcal D$ be integrable, or equivalently that $\mathcal D$ be a foliation, is equivalent to the involutivity of $T_{\mathcal D}$ under the Lie bracket of $T_X$. Conversely, given a saturated subsheaf $T_{\mathcal D}\subset T_X$, one may consider its annihilator $N^*_{\mathcal D}\subset \Omega^q_X$. After trivialising $N^*_{\mathcal D}$ on a Zariski open subset, one obtains a local generator and hence a non-zero decomposable $q$-form. This recovers the usual description of foliations by differential forms.

On complex projective space $\p^n$, foliations admit a natural numerical invariant, namely their degree. With respect to the standard polarization $\OO_{\p^n}(1)$, a codimension one foliation of degree $d$ is represented by a global section of the twisted sheaf $\Omega^1_{\p^n}(d+2)$ whose zero set has codimension at least two. The study of irreducible components of the integrable locus of $\Omega^1_{\p^n}(d+2)$ goes back to Jouanolou \cite{MR537038}, contemporaneous with the development of holomorphic foliation theory itself.

On more general projective manifolds, one often studies foliations either by using the geometry of the ambient space or by imposing additional restrictions on the foliation. We follow a more general point of view, inspired by Pereira and Perrone \cite{MR2579869}, who considered a finite-dimensional $\C$-vector space
\(
W\subset \Omega^1(\C^n,0)
\)
of germs of one-forms at the origin. They classified the subspace of integrable one-forms in $W$, under suitable assumptions on the rank of $W$ as an $\OO_{(\C^n,0)}$-submodule of $\Omega^1_{(\C^n,0)}$.

Although the possible configurations of the integrable set of such a space $W$ have been extensively studied, less attention has been paid to the geometric information that these configurations encode for the corresponding foliations. For instance, the paper \cite{MR1934363}, preceding the work of Pereira and Perrone, studies pencils of integrable foliations. Its main result, due to Cerveau, states that for any family
\(
\{\mu\omega_1+\lambda\omega_2\}
\)
of integrable one-forms on $\p^3$, indexed by $[\lambda:\mu]\in\p^1$, every member of the family either admits a singular transversely affine structure or is the pull-back, by a rational map, of a foliation on $\p^2$.

We prove a stronger statement when the whole linear space consists of integrable one-forms.

\begin{theorem}\label{T:plane int}
Let $V$ be a vector space of rational one-forms on $X$, and suppose that every element of $V$ is integrable. If $V$ has rank at least $3$, then there exists a closed rational one-form $\theta$ such that
\(
d\omega=\theta\wedge\omega
\)
for every $\omega\in V$. In particular, every foliation defined by an element of $V$ admits a singular transversely affine structure.
\end{theorem}

Although Cerveau's theorem concerns foliations on three-dimensional projective space, its methods can be adapted to arbitrary projective manifolds, with the same conclusions. This is precisely the content of \cite[Theorem A]{barbosa2025unlikely}. The central idea emphasised there is the interaction between the following three objects:
\begin{enumerate}
\item finite-dimensional vector spaces of rational one-forms on a projective manifold;
\item families of codimension one foliations containing a fixed codimension $q$ foliation;
\item transversely homogeneous structures on foliations.
\end{enumerate}

Accordingly, the present work is guided by a conjecture closely related to this problem. Although it has been attributed to several authors, it first appeared in the literature in \cite{MR2324555}.

\begin{conjecture}
Every codimension one foliation on a projective manifold $X$ either
\begin{enumerate}
\item admits a singular transversely projective structure; or
\item is the pull-back, by a rational map, of a foliation on a complex surface.
\end{enumerate}
\end{conjecture}

\section{Spaces of differential forms}
Let $W$ be a finite-dimensional vector subspace of $\Omega^1_X(U)$, the space of holomorphic one-forms on an open subset $U$ of a complex manifold $X$. We define the set of integrable one-forms in $W$ by
\[
\Int_W=\{\omega\in W\mid \omega\wedge d\omega=0\}.
\]
This condition means that $\omega$ defines a foliation on $U$. When $X$ is projective and $U$ is a Zariski open subset, the elements of $\Int_W$ define rational foliations on $X$.

\begin{proposition}
The set $\Int_W$ is an algebraic cone.
\end{proposition}

\begin{proof}
The integrability condition $\omega\wedge d\omega=0$ is quadratic in the coefficients of $\omega$. Indeed, if $\{\omega_1,\ldots,\omega_n\}$ is a basis of $W$ and
\(
\omega=\sum_{i=1}^n a_i\omega_i,
\)
then
\[
\omega\wedge d\omega=\sum_{i,j=1}^n a_i a_j\,\omega_i\wedge d\omega_j.
\]
Thus $\Int_W$ is defined by homogeneous quadratic equations in the coordinates $(a_1,\ldots,a_n)$. The number of independent equations is the dimension of the $\C$-vector space generated by the three-forms $\omega_i\wedge d\omega_j$.
\end{proof}

In order to study the foliations obtained from $\Int_W$, it is natural to remove the redundancy coming from multiplication by non-zero constants and pass to the projectivisation $\p(\Int_W)$. By the preceding proposition, $\p(\Int_W)$ is a well-defined algebraic subvariety of $\p(W)$.

It should be noted that, although two proportional one-forms define the same foliation, the set of integrable representatives may vary substantially with the choice of the ambient vector space of forms.

\begin{proposition}\label{prop:invariance_axis_firstint}
Let
\(
\{\omega_1,\ldots,\omega_q\}\subset H^0(X,\Omega^1_X\otimes\Mer_X)
\)
be integrable one-forms such that
\(
\omega_1\wedge\cdots\wedge\omega_q\neq0
\)
and such that the sum
\(
\omega_1+\cdots+\omega_q
\)
is also integrable. Suppose that there exist non-zero rational functions $f_1,\ldots,f_q$ such that the one-form
\(
f_1\omega_1+\cdots+f_q\omega_q
\)
is integrable. Then, for every pair of indices $i,j$, one has
\begin{equation}\label{eq:1st int quotient}
 d(f_i/f_j)\wedge\omega_1\wedge\cdots\wedge\omega_q=0.
\end{equation}
\end{proposition}

\begin{proof}
Set
\(
\omega=\sum_{k=1}^q f_k\omega_k.
\)
Then
\[
 d\omega=\sum_{k=1}^q\bigl(df_k\wedge\omega_k+f_k d\omega_k\bigr),
\]
and hence
\begin{equation}\label{eq:integrability_omega_functions}
\omega\wedge d\omega
=\sum_{i,j} f_i df_j\wedge\omega_i\wedge\omega_j
 +\sum_{i,j} f_i f_j\,\omega_i\wedge d\omega_j.
\end{equation}
Let
\(
\Omega=\omega_1\wedge\cdots\wedge\omega_q.
\)
Fix two distinct indices $i,j$, and denote by $\Omega_{ij}$ the wedge product obtained from $\Omega$ by omitting $\omega_i$ and $\omega_j$; if $q=2$, this is understood as the empty wedge, equal to $1$. Wedging \eqref{eq:integrability_omega_functions} with $\Omega_{ij}$ gives
\[
0=\omega\wedge d\omega\wedge\Omega_{ij}
=(f_i df_j-f_j df_i)\wedge\Omega
+ f_i f_j(\omega_i\wedge d\omega_j+\omega_j\wedge d\omega_i)\wedge\Omega_{ij}.
\]
If all functions $f_k$ are set equal to $1$, the same computation applied to the integrable form
\(
\omega_1+\cdots+\omega_q
\)
shows that the second summand vanishes. Therefore
\[
(f_i df_j-f_j df_i)\wedge\Omega=0.
\]
Since
\(
f_i df_j-f_j df_i=f_i^2d(f_j/f_i),
\)
the asserted identity follows.
\end{proof}

A rational function $f$ satisfying $df\wedge\omega=0$ for a $q$-form $\omega$ is called a first integral of the codimension $q$ foliation defined by $\omega$.

If a codimension one foliation $\F_i$ is defined by $\omega_i$, then the non-vanishing wedge product
\(
\omega=\omega_1\wedge\cdots\wedge\omega_q
\)
defines a codimension $q$ foliation. Moreover, condition \eqref{eq:1st int quotient} is precisely the assertion that $f_i/f_j$ is a first integral of this codimension $q$ foliation. When the field of first integrals is reduced to $\C$, the preceding proposition gives the following consequence.

\begin{corollary}
Let $V\subset H^0(X,\Omega^1_X\otimes\Mer_X)$ be a $\C$-subspace of dimension $q$, and let $\G$ be the foliation defined by $V$. If $\p(\Int_V)$ consists of $q+1$ points and the field of first integrals of $\G$ is $\C$, then there is a bijection between the set of codimension one foliations containing $\G$ and $\p(\Int_V)$.
\end{corollary}

\begin{example}\label{ex:pencil}
Let $\omega_1,\omega_2$ be rational integrable one-forms on a projective manifold $X$. If $\omega_1+\omega_2$ is integrable, then every element of the pencil
\(
\mu\omega_1+\lambda\omega_2
\), with $[\mu:\lambda]\in\p^1$, is integrable. Indeed,
\[
(\mu\omega_1+\lambda\omega_2)\wedge d(\mu\omega_1+\lambda\omega_2)
=
\mu\lambda(\omega_1+\omega_2)\wedge d(\omega_1+\omega_2).
\]
Consequently, if three points of a line $\ell\subset\p(W)$ belong to $\p(\Int_W)$, then the whole line $\ell$ is contained in $\p(\Int_W)$.
\end{example}

Motivated by Example~\ref{ex:pencil}, one may ask the following question: given a family of distributions parametrised by a rational curve
\(
\varphi:\p^1\to\p(W),
\)
how many integrable members are needed in order to force the integrability of the whole family?

This question has been answered for Veronese curves. These are the embeddings of $\p^1$ into $\p^k$ which, in suitable coordinates, are given by
\begin{align*}
\p^1 &\hookrightarrow \p^k,\\
[x:y]&\longmapsto [x^k:x^{k-1}y:\cdots:y^k].
\end{align*}

The notion of a Veronese web was introduced by Gelfand and Zakharevich in \cite{GELFAND1991150}. It consists of a family, parametrised by a Veronese curve, of Frobenius-integrable distributions of codimension one. In the terminology used here, a Veronese web may be viewed as the image of a Veronese curve contained in $\p(\Int_W)$. Subsequently, Panasyuk \cite{zbMATH02064118} introduced the notion of a generalized Veronese distribution, in which all distributions have fixed codimension $k$ rather than codimension one. He proved the following result.

\begin{theorem}\label{T:Veronese curves}
Let $M^{k(n+1)}$ be a manifold, and let $\{\mathfrak w(t)\}_{t\in\mathbb{RP}^1}$ be a generalized Veronese curve of distributions of codimension $k$ on $M$. In order that $\{\mathfrak w(t)\}_{t\in\p^1}$ be a generalized Veronese web, it is sufficient that it be integrable at $n+3$ distinct points $t_0,\ldots,t_{n+2}\in\p^1$.
\end{theorem}

See \cite[Theorem 2]{zbMATH02064118} for the proof; see also \cite[Theorem 4.1]{MR2268534}.

\begin{example}[Integrable sets in the projective plane]
Suppose that $W$ is three-dimensional. We record the possible configurations of the Zariski-closed set $\p(\Int_W)\subset\p(W)$.

If $\dim\p(\Int_W)=0$, then $\p(\Int_W)$ consists of at most four points. If five points of the plane belong to the integrable set, then they determine a unique conic, and Theorem~\ref{T:Veronese curves} implies that this conic is contained in the integrable set.

If $\dim\p(\Int_W)=1$, then $\p(\Int_W)$ is either a line, a possibly singular conic, or the union of a line and an isolated point. Recall that a singular conic is the union of two lines.

More generally, if $\p(\Int_W)$ contains a curve of degree $d\geq3$, then a general line intersects this curve in at least three points and must therefore be entirely contained in $\p(\Int_W)$. It follows that $\p(\Int_W)=\p(W)$. The same argument applies if $\p(\Int_W)$ contains a conic together with an additional point.
\end{example}

\subsection{Godbillon--Vey sequences}
Let $\F$ be a foliation on a complex manifold $X$. A Godbillon--Vey sequence for $\F$ is a sequence of meromorphic one-forms
\(
(\omega_0,\omega_1,\ldots)
\)
on $X$ such that $\F$ is defined by $\omega_0$ and the one-form
\begin{equation}\label{eq:GodbillonVey_def}
\Omega=dz+\sum_{n\geq0}z^n\omega_n
\end{equation}
is given by a convergent series and is integrable on $X\times\C$.

The length of a Godbillon--Vey sequence $(\omega_0,\omega_1,\ldots)$ is the least positive integer $n$ such that $\omega_k=0$ for every $k\geq n$. If no such integer exists, the sequence is said to have infinite length.

For each $t\in\C$, consider the pull-back of $\Omega$ by the inclusion
\begin{align*}
i_t:X&\longrightarrow X\times\C,\\
x&\longmapsto(x,t).
\end{align*}
The family $\{i_t^*\Omega\}_{t\in\C}$ consists of integrable one-forms on $X$, parametrised by $\C$. If the Godbillon--Vey sequence has finite length, this construction produces normal rational curves inside $\p(\Int_W)$, where $W$ is the $\C$-vector space spanned by the one-forms in the sequence. In length one and length two one obtains, respectively, lines and conics. These cases are closely related to the transverse structures of foliations.

Let now $\F$ be a codimension one singular foliation on a projective manifold $X$, defined by a rational one-form $\omega_0$. We shall say that $\F$ admits a singular transversely affine, respectively transversely projective, structure if there exists a Godbillon--Vey sequence for $\F$ of length at most $2$, respectively at most $3$. In this way, the projective geometry of $\p(\Int_W)$ records the transverse geometry of the foliations represented by the elements of $W$. Consequently, we obtain the following hierarchy for codimension one foliations: those defined by an exact $1$-form, those defined by a closed $1$-form, those admitting a transversely affine structure, and those admitting a transversely projective structure. These types of foliations admit special kinds of first integrals (e.g., multivalued or liouvillian). When nonsingular and on a simply connected manifold, these notions coincide; see \cite{MR1432053}.

\section{Structure of $\p(\Int_W)$}
Let
\(
W\subset H^0(X,\Omega^1_X\otimes\Mer_X)
\)
be a finite-dimensional $\C$-subspace. Its rank is the dimension of its linear span over the field $\C(X)$; equivalently, $\operatorname{rank}(W)$ is the greatest integer $r$ for which the natural map
\[
\bigwedge^r W\longrightarrow H^0(X,\Omega^r_X\otimes\Mer_X)
\]
induced by inclusion is non-zero. In particular, $W$ defines a codimension $r$ distribution on $X$.

The closer the rank of $W$ is from its dimension, the more rigid the components of $\Int_W$ become. In particular, \cite[Theorem 1]{MR2579869} states that if $I$ is an irreducible component of $\p(\Int_W)$ whose codimension in $\p(W)$ is at most $\operatorname{rank}(W)-2$, then $I$ is a variety of minimal degree. Recall that a variety $Y\subset\p^n$ is of minimal degree if
\[
\deg(Y)=\dim\operatorname{span}(Y)-\dim Y+1,
\]
where $\operatorname{span}(Y)$ denotes the linear span of $Y$ in $\p^n$.

Moreover, in the special case where the rank equals the dimension, \cite[Theorem 2]{MR2579869} asserts that, if $\operatorname{rank}(W)=\dim W$, then every irreducible component of $\Int_W$ is either a linear subspace or a rational curve in its linear span.

\subsection{Curvature of 3-webs}
Let $\omega$ be a rational one-form defining a codimension one foliation $\F$. Since $\omega\wedge d\omega=0$, there exists a rational one-form $\theta$ such that
\[
d\omega=\theta\wedge\omega.
\]

For another choice of $1$-form defining $\F$, say $\omega'=f\omega$ we have that
\[
    d\omega' = df\wedge\omega + fd\omega = \left(\frac{df}{f} + \theta\right) \wedge \omega'.
\]

More generally, let
\(
\omega_0,\omega_1,\omega_2\in H^0(X,\Omega^1_X\otimes\Mer_X)
\)
be integrable rational one-forms satisfying
\[
\omega_0+\omega_1+\omega_2=0.
\]
Then there exists a unique rational one-form $\theta$ such that
\[
d\omega_i=\theta\wedge\omega_i,\qquad i=0,1,2.
\]

The foliations defined by the $\omega_i$ determine a $3$-web
\[
\mathcal W=\F_1\boxtimes\F_2\boxtimes\F_3.
\]
The two-form $\Theta=d\theta$ is the curvature of the web and depends only on $\mathcal W$. If $\Theta\neq0$, then $\Theta$ defines the codimension two foliation
\(
\G=\bigcap_{i=1}^3\F_i,
\)
also defined by the $2$-form $\omega_0\wedge\omega_1$.
If $\Theta=0$, then each foliation in the web is transversely affine.

\subsection{Linear spaces}
We now prove Theorem~\ref{T:plane int}.

\begin{lemma}\label{l:theta pencil}
Let $\alpha$ and $\beta$ be rational one-forms with $\alpha\wedge\beta\neq0$. Assume that every element of the pencil $\alpha+t\beta$, $t\in\C$, is integrable. Then there exists a unique rational one-form $\theta_{\alpha\beta}$ such that
\[
d\alpha=\theta_{\alpha\beta}\wedge\alpha,
\qquad
 d\beta=\theta_{\alpha\beta}\wedge\beta.
\]
\end{lemma}
This lemma was originally proved in \cite{MR1934363}. As the same basic arguments reappear later, we include the proof here for completeness.

\begin{proof}
Since $\alpha$ and $\beta$ are integrable, there exist rational one-forms $\theta_\alpha$ and $\theta_\beta$ such that
\(
d\alpha=\theta_\alpha\wedge\alpha
\)
and
\(
d\beta=\theta_\beta\wedge\beta.
\)
The integrability of $\alpha+\beta$ gives
\[
\alpha\wedge d\beta+\beta\wedge d\alpha=0.
\]
Substituting the preceding expressions, one obtains
\[
(\theta_\beta-\theta_\alpha)\wedge\alpha\wedge\beta=0.
\]
Thus
\(
\theta_\beta-\theta_\alpha=a\alpha+b\beta
\)
for some rational functions $a,b$. Setting
\(
\theta_{\alpha\beta}=\theta_\alpha+a\alpha=\theta_\beta-b\beta
\)
gives the required one-form. If two one-forms have the stated property, their difference wedges trivially with both $\alpha$ and $\beta$; since $\alpha\wedge\beta\neq0$, the difference is zero. This proves uniqueness.
\end{proof}

\begin{proposition}\label{P:plane of integrables}
Let
\(
W=\C\omega_0\oplus\cdots\oplus\C\omega_q
\)
be a vector space of rational one-forms on $X$ such that $\operatorname{rank}(W)\geq3$ and $\Int_W=W$. Then there exists a closed rational one-form
\(
\theta\in H^0(X,\Omega^1_X\otimes\Mer_X)
\)
satisfying
\[
\theta\wedge\omega=d\omega
\qquad\text{for every }\omega\in W.
\]
\end{proposition}

\begin{proof}
Choose three elements $\omega_0,\omega_1,\omega_2\in W$ such that
\(
\omega_0\wedge\omega_1\wedge\omega_2\neq0.
\)
For each pair of distinct indices $i,j\in\{0,1,2\}$, the pencil
\(
\omega_i+t\omega_j
\)
consists of integrable one-forms. By Lemma~\ref{l:theta pencil}, there exists a unique rational one-form $\theta_{ij}$ such that
\[
\theta_{ij}\wedge\omega_i=d\omega_i,
\qquad
\theta_{ij}\wedge\omega_j=d\omega_j.
\]

For distinct $i,j,k$, the identity
\[
(\theta_{ij}-\theta_{ki})\wedge\omega_i=0
\]
implies that $\theta_{ij}-\theta_{ki}$ is proportional to $\omega_i$. Thus there are rational functions $K_0,K_1,K_2$ such that
\[
\begin{aligned}
\theta_{01}-\theta_{20}&=K_0\omega_0,\\
\theta_{12}-\theta_{01}&=K_1\omega_1,\\
\theta_{20}-\theta_{12}&=K_2\omega_2.
\end{aligned}
\]
Summing these equations gives
\[
K_0\omega_0+K_1\omega_1+K_2\omega_2=0.
\]
Since $\omega_0,\omega_1,\omega_2$ are linearly independent over the vector space of rational functions of $X$, it follows that $K_0=K_1=K_2=0$. Hence
\(
\theta_{01}=\theta_{12}=\theta_{20}
\);
we denote this common form by $\theta$.

We next show that the same $\theta$ works for every element of $W$. Let $\eta\in W$ be non-zero. Since $\omega_0,\omega_1,\omega_2$ have rank $3$, at least two of them, say $\omega_i$ and $\omega_j$, are such that
\(
\omega_i\wedge\omega_j\wedge\eta\neq0
\), unless $\eta$ is already a constant linear combination of $\omega_i$ and $\omega_j$, in which case the assertion follows by linearity. Applying the preceding argument to the triple $(\omega_i,\omega_j,\eta)$, and using the uniqueness in Lemma~\ref{l:theta pencil} for the common pencil spanned by $\omega_i$ and $\omega_j$, we obtain
\(
d\eta=\theta\wedge\eta.
\)
Thus $\theta\wedge\omega=d\omega$ for every $\omega\in W$.

It remains to prove that $\theta$ is closed. From $d\omega_i=\theta\wedge\omega_i$ and $d^2\omega_i=0$, one obtains
\[
d\theta\wedge\omega_i=0,
\qquad i=0,1,2.
\]
Therefore $d\theta$ may be written both as
\(
a_{01}\omega_0\wedge\omega_1
\)
and as
\(
a_{12}\omega_1\wedge\omega_2
\)
for suitable rational functions $a_{01},a_{12}$. Wedging the first expression with $\omega_2$ gives
\[
a_{01}\omega_0\wedge\omega_1\wedge\omega_2=0.
\]
Since the triple wedge is non-zero, $a_{01}=0$, and therefore $d\theta=0$.
\end{proof}

Proposition~\ref{P:plane of integrables} proves Theorem~\ref{T:plane int}. It also shows that the passage from pencils to higher-dimensional linear spaces imposes a strong transverse affine rigidity. This does not mean, however, that any of the foliations defined by elements in $W$ themselves must be defined by a closed one-form; the following example shows that the general conclusion cannot be strengthened in that direction.

\begin{example}\label{ex:plano de integraveis}
Consider the following one-forms, defined in an affine chart of $\p^1\times\cdots\times\p^1$:
\[
\omega_i=\frac{x^3}{3}\,dy_i+(x+y_i)\,dx,
\qquad
\eta=\left(\frac{3}{x}-\frac{3}{x^3}\right)dx.
\]
Then $\eta$ is closed and, for every $i\leq q$, one has
\(
d\omega_i=\eta\wedge\omega_i.
\)
If $W$ is the $\C$-vector space spanned by the $\omega_i$, then $\Int_W=W$. Nevertheless, none of the one-forms $\omega_i$ admits an integrating factor. Thus these foliations admit transversely affine structures, but they are not defined by closed one-forms, see \cite[Example 2.8]{MR3294560}.
\end{example}

\subsection{Other integrable sets}
When studying different configurations for $\mathbb{P}(\operatorname{Int}_W)$, the $1$-forms $\theta_t$ satisfying $\theta_t \wedge \omega_t = d\omega_t$ typically vary with $\omega_t$. To proceed analogously to the case of linear spaces, we consider the foliation $\G$ defined by $W$. For any vector field $v$ tangent to $\G$, we obtain
\[
\mathcal{L}_v \omega_t = \theta_t(v)\,\omega_t.
\]
Whenever a collection of $1$-forms in $\mathbb{P}(\operatorname{Int}_W)$ satisfies a nontrivial linear dependence relation, the values $\theta_t(v)$ become independent of $t$. Equivalently, the restrictions $\theta_t|_{T\G}$ coincide for all such $t$. In this situation, the partial connection
\[
\nabla: N^*_{\G} \otimes \Mer_X \longrightarrow \hom(T\G, N^*_{\G} \otimes \Mer_X),\qquad \omega \longmapsto \mathcal{L}_v(\omega).
\]
known as the Bott connection for the conormal bundle of $\G$, becomes diagonalizable. As a special consequence of the study of Bott connections whose connection matrix is a multiple of the identity, we have \cite[Theorem B]{barbosa2025unlikely}. From this result, we obtain the following corollary.

\begin{theorem}\label{T:4points}
    Let $X$ be a projective manifold of dimension at least four.
    Let $W$ be a three-dimensional vector space of $1$-forms on $X$ and suppose that $\p(\operatorname{Int}_W)$ has at least four points in general position. Then every foliation defined by a element of $\p(\operatorname{Int}_W)$ either admits a transversely projective structure, or is a pullback from a foliation on a projective manifold $Y$, where $\dim Y \le 3$.
\end{theorem}

Theorem~\ref{T:4points} applies both to the case where $\mathbb{P}(\operatorname{Int}_W)$ consists of $4$ points and to the case where it is an irreducible conic. When $\dim X = 3$, although any integrable $1$-form can be arranged into a configuration of $4$ points, it remains unknown what kind of properties one obtains in the conic case.

\subsection{Bihamiltonian structures}
A Poisson structure on a manifold $X$ is a bivector field
\(
\Pi\in H^0(X,\wedge^2T_X)
\)
whose Schouten bracket with itself vanishes:
\(
[\Pi,\Pi]=0.
\)

On a projective manifold \(X\), the space $H^0(X,\wedge^2 T_X)$ of global bivector fields is finite-dimensional as a \(\mathbb{C}\)-vector space. In this setting, one can carry out an analogous study of the irreducible components of \(H^0(X,\wedge^2 T_X)\). Moreover, although the Schouten bracket behaves similarly to the operator
\[
(\omega_1,\omega_2)\longmapsto \omega_1\wedge d\omega_2+\omega_2\wedge d\omega_1
\]
on any space of $1$-forms, and so we get natural analogues for Examples \ref{ex:pencil} and \ref{ex:plano de integraveis}, it is not true in general that the product of a Poisson bivector with a function is again Poisson. Indeed, if \(\Pi\) is a Poisson bivector and \(f\in H^0(X,\mathcal O_X)\), then
\[
[f\Pi,f\Pi]
=
2f\,\Pi^\sharp(df)\wedge \Pi
+
f^2[\Pi,\Pi],
\]
where
\[
\Pi^\sharp:\Omega_X^1\longrightarrow T_X,
\qquad
\alpha\longmapsto \iota_\alpha \Pi
\]
is the contraction map induced by \(\Pi\). Since \([\Pi,\Pi]=0\), we obtain
\[
[f\Pi,f\Pi]
=
2f\,\Pi^\sharp(df)\wedge \Pi,
\]
which does not vanish in general. Hence \(f\Pi\) is not necessarily a Poisson bivector.

\begin{example}
When $X$ is a complex $3$-fold, a bivector $\pi \in H^0(X,\wedge^2 T_X)$ corresponds, under the isomorphism
\[
\wedge^2 T_X \cong \Omega^1_X \otimes \omega_X^{-1},
\]
to a twisted $1$-form
\(
\eta \in H^0(X,\Omega^1_X \otimes \omega_X^{-1}).
\)
The Poisson condition $[\Pi,\Pi]=0$ is equivalent to the integrability
condition $\eta_U \wedge d\eta_U = 0$ for $\eta$ in local trivializations.
Consequently, $\eta$ defines a codimension one holomorphic foliation
$\F$ whose normal bundle is isomorphic to $\omega_X^{-1}$.
From the exact sequence
\[
0 \to T_{\F} \to T_X \to N_{\F} \to 0,
\]
we obtain
\[
K_{\F} = K_X \otimes N_{\F} \cong \omega_X \otimes \omega_X^{-1} \cong \mathcal{O}_X.
\]
\end{example}

\subsection*{Poisson pencils}
Consider two Poisson bivectors $\Pi_1,\Pi_2\in H^0(X,\wedge^2 T_X)$. Their sum $\Pi_1+\Pi_2$ is again a Poisson bivector if and only if $[\Pi_1,\Pi_2]=0$. As illustrated in Example~\ref{ex:pencil}, the fact that three distinct Poisson structures are collinear is equivalent to the entire pencil being integrable. Such pencils are known in the literature as bihamiltonian systems.
It was proved by Gelfand and Zakharevich \cite{GELFAND1991150} that a general bihamiltonian system on a $(2k+1)$-fold corresponds to a Veronese web on the $k$-dimensional germ $(\C^{k+1},0)$. Here, generic means that the Poisson pencil has maximal rank $2k$, since Poisson structures have even-dimensional symplectic leaves.
In particular, for a Poisson pencil with maximal rank on a $5$-fold we obtain a Veronese web $\omega_0+t\omega_1+t^2\omega_2$ on a germ $(\C^3,0)$.

\subsection*{Acknowledgements}

The author is grateful to Jorge Vit\'orio Pereira for his guidance and for numerous insightful conversations. The author also acknowledges the support of the CNPq postdoctoral fellowship.

\bibliographystyle{amsplain}
\bibliography{references}

@article {MR2579869,
    AUTHOR = {Pereira, Jorge Vit\'orio and Perrone, Carlo},
     TITLE = {Germs of integrable forms and varieties of minimal degree},
   JOURNAL = {Bull. Sci. Math.},
  FJOURNAL = {Bulletin des Sciences Math\'ematiques},
    VOLUME = {134},
      YEAR = {2010},
    NUMBER = {1},
     PAGES = {1--11},
      ISSN = {0007-4497,1952-4773},
   MRCLASS = {32S65},
  MRNUMBER = {2579869},
MRREVIEWER = {M.\ G.\ Soares},
       DOI = {10.1016/j.bulsci.2009.09.005},
       URL = {https://doi.org/10.1016/j.bulsci.2009.09.005},
}

@article {MR1934363,
    AUTHOR = {Cerveau, Dominique},
     TITLE = {Pinceaux lin\'eaires de feuilletages sur {$CP(3)$}
              et conjecture de {B}runella},
   JOURNAL = {Publ. Mat.},
  FJOURNAL = {Publicacions Matem\`atiques},
    VOLUME = {46},
      YEAR = {2002},
    NUMBER = {2},
     PAGES = {441--451},
      ISSN = {0214-1493,2014-4350},
   MRCLASS = {32S65 (37F75)},
  MRNUMBER = {1934363},
MRREVIEWER = {Jorge\ Vit\'orio\ Pereira},
       DOI = {10.5565/PUBLMAT\_46202\_06},
       URL = {https://doi.org/10.5565/PUBLMAT_46202_06},
}

@article{zbMATH02064118,
 author = {Panasyuk, Andriy},
 title = {On integrability of generalized {Veronese} curves of distributions},
 fjournal = {Reports on Mathematical Physics},
 journal = {Rep. Math. Phys.},
 issn = {0034-4877},
 volume = {50},
 number = {3},
 pages = {291--297},
 year = {2002},
 language = {English},
 doi = {10.1016/S0034-4877(02)80059-3},
 zbMATH = {2064118},
 Zbl = {1042.53008}
}

@article{GELFAND1991150,
TITLE = {Webs, Veronese curves, and bihamiltonian systems},
JOURNAL = {Journal of Functional Analysis},
VOLUME = {99},
NUMBER = {1},
PAGES = {150-178},
YEAR = {1991},
ISSN = {0022-1236},
DOI = {https://doi.org/10.1016/0022-1236(91)90057-C},
URL = {https://www.sciencedirect.com/science/article/pii/002212369190057C},
AUTHOR = {Israel M Gelfand and Ilya Zakharevich},
}

@article{barbosa2025unlikely,
  title={Unlikely intersections of codimension one foliations},
  author={Barbosa, Gabriel Santos and Pereira, Jorge Vit{\'o}rio},
  journal={arXiv preprint arXiv:2505.14873},
  year={2025}
}

@article {MR2324555,
    AUTHOR = {Cerveau, Dominique and Lins-Neto, Alcides and Loray, Frank and
              Pereira, Jorge Vit\'orio and Touzet, Fr\'ed\'eric},
     TITLE = {Complex codimension one singular foliations and
              {G}odbillon-{V}ey sequences},
   JOURNAL = {Mosc. Math. J.},
  FJOURNAL = {Moscow Mathematical Journal},
    VOLUME = {7},
      YEAR = {2007},
    NUMBER = {1},
     PAGES = {21--54, 166},
      ISSN = {1609-3321,1609-4514},
   MRCLASS = {37F75 (32S65)},
  MRNUMBER = {2324555},
MRREVIEWER = {Yohann\ Genzmer},
       DOI = {10.17323/1609-4514-2007-7-1-21-54},
       URL = {https://doi.org/10.17323/1609-4514-2007-7-1-21-54},
}

@article {MR2268534,
    AUTHOR = {Bouetou, Thomas B. and Dufour, Jean P.},
     TITLE = {Veronese curves and webs: interpolation},
   JOURNAL = {Int. J. Math. Math. Sci.},
  FJOURNAL = {International Journal of Mathematics and Mathematical
              Sciences},
      YEAR = {2006},
     PAGES = {Art. ID 93142, 11},
      ISSN = {0161-1712,1687-0425},
   MRCLASS = {53A60},
  MRNUMBER = {2268534},
MRREVIEWER = {Vladislav\ Goldberg},
       DOI = {10.1155/IJMMS/2006/93142},
       URL = {https://doi.org/10.1155/IJMMS/2006/93142},
}

@article {MR3294560,
    AUTHOR = {Cousin, Ga\"el and Pereira, Jorge Vit\'orio},
     TITLE = {Transversely affine foliations on projective manifolds},
   JOURNAL = {Math. Res. Lett.},
  FJOURNAL = {Mathematical Research Letters},
    VOLUME = {21},
      YEAR = {2014},
    NUMBER = {5},
     PAGES = {985--1014},
      ISSN = {1073-2780,1945-001X},
   MRCLASS = {57R30 (32S65 53C12)},
  MRNUMBER = {3294560},
MRREVIEWER = {Cristian\ Ida},
       DOI = {10.4310/MRL.2014.v21.n5.a5},
       URL = {https://doi.org/10.4310/MRL.2014.v21.n5.a5},
}

@article {MR1432053,
    AUTHOR = {Sc\'ardua, Bruno Azevedo},
     TITLE = {Transversely affine and transversely projective holomorphic
              foliations},
   JOURNAL = {Ann. Sci. \'Ecole Norm. Sup. (4)},
  FJOURNAL = {Annales Scientifiques de l'\'Ecole Normale Sup\'erieure.
              Quatri\`eme S\'erie},
    VOLUME = {30},
      YEAR = {1997},
    NUMBER = {2},
     PAGES = {169--204},
      ISSN = {0012-9593},
   MRCLASS = {32L30},
  MRNUMBER = {1432053},
MRREVIEWER = {M.\ G.\ Soares},
       DOI = {10.1016/S0012-9593(97)89918-1},
       URL = {https://doi.org/10.1016/S0012-9593(97)89918-1},
}

@book {MR537038,
    AUTHOR = {Jouanolou, Jean-Pierre},
     TITLE = {\'Equations de {P}faff alg\'ebriques},
    SERIES = {Lecture Notes in Mathematics},
    VOLUME = {708},
 PUBLISHER = {Springer, Berlin},
      YEAR = {1979},
     PAGES = {v+255},
      ISBN = {3-540-09239-0},
   MRCLASS = {14D05 (32C40 34A20 35A20 58A17)},
  MRNUMBER = {537038},
MRREVIEWER = {V.\ A.\ Golubeva},
}
\end{document}